\documentclass[twoside,11pt,a4paper]{amsart}
\usepackage{amssymb,supertabular,a4}
\usepackage{amsmath,amsthm,amscd,amssymb,graphicx}
\input xy
\xyoption{all}
\def\Hom{\hbox{\rm Hom\,}}
\def\Ext{\hbox{\rm Ext\,}}

\def\dimHom{\hbox{\rm dimHom\,}}
\def\dimExt{\hbox{\rm dimExt\,}}

\def\mod{\hbox{\rm mod\,}}

\def\ql{\hbox{\rm ql\,}}
\def\qs{\hbox{\rm qs\,}}
\def\qt{\hbox{\rm qt\,}}
\def\udim{\hbox{\rm \underline{dim}\,}}
\def\ra{{\rightarrow}}

\title[Regular modules with preprojective Gabriel-Roiter submodules]
{Regular modules with preprojective Gabriel-Roiter submodules over
$n$-Kronecker quivers}
\author{Bo Chen}
\email {mcebbchen@googlemail.com}
\address {Universit\"at zu K\"oln\\
          Mathematisches Institut\\
                   Weyertal 86-90\\
           D-50931 K\"oln\\ Germany}


\thanks{The author is supported by DFG-Schwerpunktprogramm 1388
`Darstellungstheorie'.}

\newtheorem{theo}{Theorem}[section]
\newtheorem{defi}[theo]{Definition}
\newtheorem{lemm}[theo]{Lemma}
\newtheorem{coro}[theo]{Corollary}
\newtheorem{prop}[theo]{Proposition}

\newtheorem*{theo1}{Theorem 1}
\newtheorem*{theo2}{Theorem 2}
\newtheorem*{theo3}{Theorem 3}
\newtheorem*{theo4}{Theorem 4}

\newenvironment{example}[1][Example]{\begin{trivlist}
\item[\hskip \labelsep {\bfseries #1}]}{\end{trivlist}}
\newenvironment{remark}[1][Remark]{\begin{trivlist}
\item[\hskip \labelsep {\bfseries #1}]}{\end{trivlist}}

\begin{document}

\begin{abstract} Let $Q$ be a wild $n$-Kronecker quiver, i.e., a quiver with two vertices,
labeled by $1$ and $2$, and $n\geq 3$ arrows from $2$ to $1$. The
indecomposable regular modules with preprojective Gabriel-Roiter
submodules, in particular, those $\tau^{-i}X$ with $\udim X=(1,c)$
for $i\geq 0$ and some $1\leq c\leq n-1$ will be studied.  It will be shown that for
each $i\geq 0$ the irreducible monomorphisms starting with $\tau^{-i}X$ give
rise to a sequence of Gabriel-Roiter inclusions, and moreover, the
Gabriel-Roiter measures of those produce a sequence of direct successors.
In particular, there are infinitely many GR-segments, i.e., a sequence of
Gabriel-Roiter measures closed under direct successors and predecessors.
The case $n=3$ will be studied in detail with the help
of Fibonacci numbers.
It will be  proved that for a regular component containing some indecomposable
module with dimension vector $(1,1)$ or $(1,2)$, the Gabriel-Roiter
measures of the indecomposable modules are uniquely determined by their dimension vectors.
\end{abstract}

\maketitle


\section{Introduction}

Let $\Lambda$ be an artin algebra and $\mod\Lambda$ the category of
finitely generated right $\Lambda$-modules. For each
$M\in\mod\Lambda$, we denote by $|M|$ the length of $M$. The symbol
$\subset$ is used to denote proper inclusion.

We first recall the original definition of Gabriel-Roiter measure
\cite{R3,R4}. Let $\mathbb{N}$=$\{1,2,\ldots\}$ be the set of
natural numbers and $\mathcal{P}(\mathbb{N})$ be the set of all
subsets of $\mathbb{N}$.  A total order on
$\mathcal{P}(\mathbb{N})$ can be defined as follows: if $I$,$J$
are two different subsets of $\mathbb{N}$, write $I<J$ if the
smallest element in $(I\backslash J)\cup (J\backslash I)$ belongs to
J. Also we write $I\ll J$ provided $I\subset J$ and for all elements
$a\in I$, $b\in J\backslash I$, we have $a<b$. We say that $J$ {\bf
starts with} $I$ if $I=J$ or $I\ll J$. Thus $I<J<I'$ with $I'$
starts with $I$ implies that $J$ starts with $I$.

For each $M\in\mod\Lambda$, let $\mu(M)$ be the maximum of the sets
$\{|M_1|,|M_2|,\ldots, |M_t|\}$, where $M_1\subset M_2\subset \ldots
\subset M_t$ is a chain of indecomposable submodules of $M$. We call
$\mu(M)$ the {\bf Gabriel-Roiter} ({\bf GR} for short) {\bf measure}  of $M$. A
subset $I$ of $\mathcal{P}(\mathbb{N})$ is called a GR measure for $\Lambda$ if
there is an indecomposable $\Lambda$-module $M$ with $\mu(M)=I$. If $M$ is an
indecomposable $\Lambda$-module, we call an inclusion $X\subset M$
with $X$ indecomposable a {\bf GR inclusion} provided
$\mu(M)=\mu(X)\cup\{|M|\}$, thus if and only if every proper
submodule of $M$ has Gabriel-Roiter measure at most $\mu(X)$. In
this case, we call $X$ a {\bf GR submodule} of $M$.  Note that the
factor of a GR inclusion is indecomposable.

Using Gabriel-Roiter  measure, Ringel obtained a partition of the
module category for any artin algebra of infinite representation
type \cite{R3,R4}: there are infinitely many GR measures $I_i$ and
$I^i$ with $i$ natural numbers, such that
$$I_1<I_2<I_3<\ldots\quad \ldots <I^3<I^2<I^1$$ and such that any
other GR measure $I$ satisfies $I_i<I<I^j$ for all $i,j$. The GR
measures $I_i$ (resp. $I^i$) are called take-off (resp. landing)
measures. Any other GR measure is called a central measure. An
indecomposable module $M$ is  called a  take-off (resp. central,
landing) module if its GR measure $\mu(M)$ is a take-off (resp.
central, landing) measure. It was proved in \cite{R3} that every landing
module is preinjective in general sense.

Let $I,I'$ be two GR measures for $\Lambda$. We call $I'$  a {\bf
direct successor} of $I$ if, first, $I<I'$ and second, there does
not exist a GR measure $I''$ with $I<I''<I'$. The so-called {\bf
Successor Lemma} in \cite{R4} states that any GR measure $I$
different from $I^1$, the maximal one, has a direct successor.
However, a GR measure, which is not the minimal one $I_1$, may not admit a
direct predecessor.  A {\bf GR segment} is a sequence of GR measures
closed under direct predecessors and direct predecessors.
It was conjectured that an artin algebra if of wild type if and only if
it has infinitely many GR segments.

The GR measure for path algebras of tame quivers (over an algebraically closed field)
were studied in
\cite{Ch2, Ch3, Ch4}. In particular, the connection between  GR measure and
Auslander-Reiten theory was studied.  For example, let $\delta$
be the minimal positive imaginary root and $H_1$ be an
indecomposable homogeneous simple module (with dimension vector
$\delta$), then the sequence of irreducible monomorphisms $H_1\ra
H_2\ra H_3\ra\ldots$  gives a sequence of GR submodules.
Moreover, $\mu(H_{i+1})$ is the direct successor of $\mu(H_i)$ for
each $i\geq 1$. It was also shown in \cite{Ch4} that for a tame
quiver, there are,  but only finitely many, GR measures which have
no direct predecessors, and in \cite{Ch8} that the number of the GR segments
is bounded by $b+3$, where $b$ is the number of the isomorphism classes of exceptional
quasi-simple modules.

So far, not much about the GR measures for wild quivers is known.
In \cite{F}, it was proved for $3$-Kronecker quiver
that there are uncountable many Gabriel-Roiter measures (modules of infinite length were considered).
It was also conjectured there
the existence of the maximal central measure (which should be an infinite sequence of natural numbers).
In \cite{Ch5}, the wild $n$-Kronecker quivers were studied and infinitely many GR measures admitting no direct
predecessors were constructed.

In this paper, the study will  be focused on $n$-Kronecker
quivers:
$$\xymatrix{ 2\ar@/^1pc/[rr]^{\alpha_1}\ar@/_1pc/[rr]_{\alpha_n}&\vdots& 1\\}$$
with $n\geq 3$. The indecomposable modules whose GR submodules are
preprojective  will be studied. Similar to the case of homogeneous modules for tame quivers, the
following result will be shown:

\begin{theo1}Let $X$ be an indecomposable module containing a
preprojective module as a GR submodule. Then for each $i\geq 0$ and
each $j\geq 1$, the irreducible monomorphism
$\tau^{-i}X[j]\ra\tau^{-i}X[j+1]$ is a GR inclusion. Moreover, up to
isomorphism, $\tau^{-i}X[j]$ is the unique GR submodule of
$\tau^{-i}X[j+1]$.
\end{theo1}
Here $\tau$ is the Auslander-Reiten translation and $X[j]$ denotes
the regular module with quasi-simple submodule $X$ and quasi-length
$j$. Note that $X$ is always quasi-simple under the assumption.

The vectors $(1,c)$ with $1\leq c\leq n-1$ are imaginary roots. An
indecomposable module (existence by \cite{Ka}) $X$ with  dimension vector $(1,c)$ contains
the projective simple module as a GR submodule. Thus the above theorem
implies that $\tau^{-i}X[j+1]$ contains $\tau^{-i}X[j]$, up to
isomorphism, as the unique GR submodule.   Using some combinatorial
studies, it will be seen that the Gabriel-Roiter measures
$\mu(\tau^{-i}X[j])$ are namely determined by $i$, $j$ and the
dimension vectors $(1,c)$.

\begin{theo2}Let $X$ be an indecomposable module with dimension
$\udim X=(1,c)$ for some $1\leq c\leq n-1$ and $M$ an indecomposable
module. Let $i\geq 0$ and $j\geq 1$.  Then
$\mu(M)=\mu(\tau^{-i}X[j])$ if and only if $M\cong \tau^{-i}Y[j]$
for some indecomposable module $Y$ with $\udim Y=(1,c)=\udim X$.
\end{theo2}

As a consequence of this theorem, we may obtain the following
result, which can be used to show that the number of the GR segments
for $n$-Kronecker quivers is unbounded:

\begin{theo3} Let $X$ be an indecomposable module with dimension
vector $\udim X=(1,c)$ for some $1\leq c\leq n-1$. Then for each $i\geq
0$ and $j\geq 1$, $\mu(\tau^{-i}X[j+1])$ is the direct successor of
$\mu(\tau^{-i}X[j])$.
\end{theo3}

As an application of the general discussion, we will study in detail
the regular components  over the $3$-Kronecker quiver, which contains an indecomposable
module with dimension vector $(1,1)$ or $(1,2)$.
In \cite{Zh}, it was proved that the indecomposable modules in a regular component
of any wild hereditary algebra are uniquely determined by their dimensions.
Thus in each regular component, only finitely many indecomposable modules have the same length.
It may be asked if these modules of the same length have the same GR measure.  However, this is
not always the case (see Section \ref{anti-ex} for an example).
We can partially answer this question as follows:

\begin{theo4}Let $Q$ be the $3$-Kronecker quiver. Let $X$ be an indecomposable module with dimension
vector $(1,1)$ or $(1,2)$ and $\mathcal{C}$ a regular component containing $X$.
Then the GR measures of the indecomposable modules in $\mathcal{C}$
are uniquely determined by their dimension vectors.
\end{theo4}

In section \ref{preliminaries}, some preliminaries, notations and elementary results
will be recalled. Section \ref{preprogr} is devoted to a study of the
indecomposable regular modules with preprojective GR submodules. In
particular, the theorem concerning the coincidence of the
irreducible monomorphisms and the GR inclusions will be shown. In
Section \ref{(1,c)}, the indecomposable modules $\tau^{-i}X$ with $i\geq 0$
and $\udim X=(1,c)$ will be studied. In particular, the direct
successors of $\tau^{-i}X[j]$ will be described for $i\geq 0$ and
$j\geq 1$. The regular components containing an indecomposable module with dimension
vector $(1,1)$ or $(1,2)$ over a $3$-Kronecker quiver will be studied in detail in Section \ref{3Kronecker}.

\section{Preliminaries and known results}\label{preliminaries}

\subsection{Representations of $n$-Kronecker quivers}
 
We recall some facts of representations of quivers. The best
references are \cite{ARS,R2}. We also refer to \cite{K,R1} for
general structures of representations of wild quivers. We also refer
to \cite{Ka} for the relationship between the roots and the
indecomposable modules.

Let $Q$ be an
$n$-Kronecker quiver with $n\geq 3$ and $k$ an algebraically closed
field. A representation of $Q$ over $k$ is simply called a module.
The
Cartan matrix and the Coxeter matrix are the following: $$C=\left(\begin{array}{cc}1 & 0 \\
n & 1\\\end{array}\right),\quad \Phi=-C^{-t}C=
\left(\begin{array}{cr}n^2-1 & n \\
                        -n & -1 \\ \end{array}\right),\quad \Phi^{-1}=
\left(\begin{array}{rc}-1 & -n \\
                         n & n^2-1 \\
                         \end{array}\right).$$
The dimension vectors can be calculated using $\udim\tau M=(\udim
M)\Phi$ if $M$ is not projective and $\udim \tau^{-1}N=(\udim
N)\Phi^{-1}$ if $N$ is not injective, where $\tau$ denotes the
Auslander-Reiten translation. The quadratic form
$q((x_1,x_2))=x_1^2+x_2^2-nx_1x_2$. A vector $(a,b)$ is a real root
if $q((a,b))=1$. The positive real roots are precisely the dimension
vectors of the indecomposable preprojective modules and those of the
indecomposable preinjective modules. For each positive imaginary
root $(a,b)$, i.e., $q((a,b))<0$, there are infinitely many
indecomposable modules with dimension vector $(a,b)$. Note that the
dimension vector of an indecomposable module is either a positive
real root or a positive imaginary root.  The Euler form is $\langle
(x_1,x_2),(y_1,y_2) \rangle=x_1y_1+x_2y_2-nx_1y_2$. For two
indecomposable modules $X$ and $Y$,
$$\dimHom(X,Y)-\dimExt^1(X,Y)=\langle\udim X,\udim Y\rangle.$$

The Auslander-Reiten quiver of $Q$ consists of one preprojective
component, one preinjective component and infinitely many regular
ones. An indecomposable regular module $X$ is called quasi-simple if
the Auslander-Reiten sequence starting with $X$ has an
indecomposable middle term. For each indecomposable regular module
$M$, there is a unique quasi-simple module $X$ and a unique natural
number $r\geq 1$ (called quasi-length of $M$ and denoted by
$\ql(M)=r$) such that there is a sequence of irreducible
monomorphisms $X=X[1]\ra X[2]\ra\ldots\ra X[r]=M$. In this case, we denote by $\qs(M)=X$. 
Dually, there is a unique quasi-simple module $Y$ (denote by $\qt(M))$ with a sequence of irreducible
epimorphisms $M=[r]Y\ra\ldots\ra [2]Y\ra [1]Y=Y$.

\subsection{Properties of GR measure}
We present some known results being used later on.  The following
proposition was proved in \cite{R3}:
\begin{prop}\label{GA} Let $\Lambda$ be an artin algebra and $X$ and
$Y_1,Y_2,\ldots,Y_r$ be indecomposable modules. Assume that
$X\stackrel{f}{\ra}\oplus_{i=1}^r {Y_i}$ is a monomorphism.
\begin{itemize}
  \item[(1)] $\mu(X)\leq \textrm{max}\{\mu(Y_i)\}$.
  \item[(2)] If $\textrm{max}\{\mu(Y_i)\}=\mu(X)$,
          then $f$ splits.
\end{itemize}
\end{prop}
We collect some properties of GR inclusions in the following lemma.
The proof can be found for example in \cite{Ch1,Ch2}.
\begin{lemm}\label{GR} Let $\Lambda$ be an artin algebra and $X\subset M$ a GR
inclusion.
\begin{itemize}
  \item[(1)] If all irreducible maps to $M$ are monomorphisms, then
              the GR inclusion is an irreducible map.
  \item[(2)] Every non-zero homomorphism $Y\ra M/X$, which is not an epimorphism,
             factors through the canonical projection  $M\ra M/X$.
  \item[(3)] There is an irreducible monomorphism $X\ra Y$ with $Y$
             indecomposable and an epimorphism $Y\ra M$.
  \item[(4)] If $Y$ is indecomposable with $\mu(X)<\mu(Y)<\mu(M)$, then $|Y|>|M|$.
\end{itemize}
\end{lemm}

\subsection{The partition for $n$-Kronecker quivers}
Let $Q$ be an $n$-Kronecker quiver. We are going to describe the
partition obtained using GR measure for $Q$. 

The preprojective component is the following (note that there are
actually $n$ arrows from $P_i$ to $P_{i+1}$):
$$\xymatrix@R=8pt@C=12pt{
   &P_2=(1,n)\ar@{->}[rd]\ar@{.}[rr] &&
    P_4=(n^2-1,n^3-2n)\ar@{->}[rd]&\ldots&
     \\
  P_1=(0,1)\ar@{->}[ru]\ar@{.}[rr] &&
  P_3=(n,n^2-1)\ar@{->}[ru]\ar@{.}[rr]&&
  P_5&\ldots\\}$$
Since every irreducible map in the preprojective component is a
monomorphism, $P_i$ is, up to isomorphism, the unique GR submodule
of $P_{i+1}$ by Lemma \ref{GR}(1).
Similarly, the preinjective component is of the following form:
$$\xymatrix@R=8pt@C=12pt{
   &\ldots&Q_3=(n^3-2n,n^2-1)\ar@{->}[rd]\ar@{.}[rr] &&
    Q_1=(n,1)\ar@{->}[rd]&
     \\
  \ldots&Q_4\ar@{->}[ru]\ar@{.}[rr] &&
  Q_2=(n^2-1,n)\ar@{->}[ru]\ar@{.}[rr]&&
  Q_0=(1,0)\\}$$
Let us denote by $I_i$ (resp. $I^i$) the take-off (resp. landing) measures and
by
$\mathcal{A}(I)$ the set of the representatives (of the isomorphism classes)
of indecomposable modules with GR measure $I$.
\begin{prop}\label{part} 
\begin{itemize}
   \item[(1)]  The take-off part contains precisely
                the simple injective module and the indecomposable preprojective modules.
   \item[(2)] The landing part contains precisely all non-simple indecomposable preinjective modules.
   \item[(3)]  An indecomposable module is a central module if and only if
               it is regular.
\end{itemize}
\end{prop}

\begin{proof}
(1) We show, by induction on $m$, that $\mathcal{A}(I_m)=\{P_m\}$ for each $m\geq 2$.
If $m=2$, the assertion holds by
the description of $I_2$, which is the GR measure of a local module with maximal length.
Assume that $\mu(M)=I_{m+1}$ for some
indecomposable module $M$. Since $M$ is not simple, we may assume
that $Y$ is a GR submodule of $M$. Then $\mu(Y)=I_i\leq I_m$ for
some $i\leq m$, and thus $Y\cong P_i$ by induction. It follows from
Lemma \ref{GR}(3) that there is an epimorphism $P_{i+1}\ra M$.
In particular $|M|\leq |P_{i+1}|$. If the equality does not hold,
then
$$I_{m+1}=\mu(M)=I_i\cup\{|M|\}>I_i\cup\{|P_{i+1}|,\ldots,|P_m|,|P_{m+1}|\}
>I_i\cup\{|P_{i+1}|,\ldots,|P_m|\}=I_m.$$
This is a contradiction because the GR measure
$\mu(P_{m+1})=I_i\cup\{|P_{i+1}|,\ldots,|P_m|,|P_{m+1}|\}$ lies
between $I_m$ and $I_{m+1}$. Therefore, $|P_{i+1}|=|M|$ and thus
$P_{i+1}\cong M$. Since $\mu(M)=I_{m+1}$, we have $i=m$ and thus
$\mathcal{A}(I_{m+1})=\{P_{m+1}\}$.

(2) Since there is a short exact sequence $0\ra
I_{r+1}\ra I_r^n\ra I_{r-1}\ra 0 $ for each $r\geq 1$,
$\mu(I_{r+1})<\mu(I_r)$ by Proposition \ref{GA}. Because landing
modules are preinjective (see \cite{R3}), $\mathcal{A}(I^m)$ contains precisely
one isomorphism class $Q_m$.

(3) is straightforward.
\end{proof}

\section{Regular modules with preprojective GR submodules}\label{preprogr}

Let $Q$ be an $n$-Kronecker quiver
with $n\geq 3$.  Before studying the regular modules whose GR
submodules are preprojective, we present some combinatorial
descriptions of the indecomposable regular modules with dimension
$(a,b)$ such that $a\leq b$. We write two vectors $(a,b)<(c,d)$ if
$a<c$ and $b<d$.
\begin{lemm}\label{dim1}Let $X$ be an indecomposable regular module with dimension vector $\udim X=(a,b)$
such that $a\leq b$. Let $i\geq 1$ and assume that
$\udim\tau^{-i}X=(c,d)$. Then
\begin{itemize}
  \item[(1)] $(a,b)<(c,d)$.
  \item[(2)] $c<d$.
  \item[(3)] For each $r\geq 0$, $\sum_{i=0}^r\udim \tau^{-i}X<\udim
             \tau^{-(r+1)}X$.
\end{itemize}
\end{lemm}

\begin{proof} We show (3) and (1) and (2) follow similarly.
Let $i=1$ and we show $c-2a\geq 0$ and $d-2b>0$. Since $n\geq 3$ and
$b\geq a$, we have $c-2a=nb-3a\geq 0$. Note that the equality hold
only for $n=3$ and $a=b$. Similarly,
$d-2b=(n^2-1)b-na-2b=(n^2-3)b-na>0$. Then the proof follows by
induction.
\end{proof}

\begin{coro}\label{corodim1}
Let $X$ be a quasi-simple module with dimension vector $\udim
X=(a,b)$ and $a\leq b$.  Consider the following short exact sequence
$$0\ra \tau ^{-i}X[j]\stackrel{f}{\ra}\tau^{-i}X[j+1]\ra \tau^{-(i+j)}X\ra 0$$
where $f$ is an irreducible monomorphism and $i\geq 0$, $j\geq 1$.
Then $\udim \tau^{-i}X[j]<\udim \tau^{-(i+j)}X$ and thus
$|\tau^{-i}X[j]|<|\tau^{-(i+j)}X|$.
\end{coro}

\begin{proof}This follows directly from the  lemma with
the assumption that $X$ is quasi-simple.
\end{proof}

Let $\mathcal{B}$ be the set of the isomorphism classes of the
indecomposable regular modules whose GR submodules are
preprojective. Note that $\mathcal{B}$ is not empty since it
contains all indecomposable modules $X$ with dimension vector $\udim
X=(1,1)$. By Proposition \ref{part},  an indecomposable module $X$
is contained in $\mathcal{B}$ if and only if $X$ has no proper
regular submodules. In particular,  $X\in\mathcal{B}$ implies that
$X$ is quasi-simple.

\begin{lemm}\label{dim2} Let $X\in\mathcal{B}$  with dimension vector $\udim
X=(a,b)$. Then $a\leq b$.
\end{lemm}
\begin{proof} There is nothing to show if $a=1$.  Assume
$a\geq 2$ and let $M$ be an indecomposable regular module with
dimension vector $(1,n-1)$ (note that $M$ exists since $(1,n-1)$ is
an imaginary root). Since there does not exist an epimorphism $M\ra
X$ and $X$ has no proper regular submodules, we have $\Hom(M,X)=0$.
It follows that $\langle(1,n-1),(a,b)\rangle\leq 0$ and thus
$a-b=a+(n-1)b-nb\leq 0$.
\end{proof}

\begin{lemm}\label{B}Let $X\in\mathcal{B}$ and $Y$ be a GR submodule of $X$.
\begin{itemize}
   \item[(1)] For each $i\geq 0$, $\tau^{-i}X\in\mathcal{B}$.
   \item[(2)] There exists an $m\geq 1$ such that
   $\tau^iX\notin\mathcal{B}$ for any $i\geq m$.
   \item[(3)] If the GR factor $X/Y$ is not simple, then
              $X/Y\in\mathcal{B}$.
   \item[(4)] If $M$ is a non-simple indecomposable proper factor module of $X$, then $\mu(M)>\mu(X)$.
\end{itemize}
\end{lemm}

\begin{proof}

(1) Since a proper inclusion $M\subset\tau^{-i}X$ with $M$ a regular
module induces a proper regular submodule $\tau^iM$ of $X$, $\tau^{-i}X$ has
no proper regular submodules and thus $\tau^{-i}X\in\mathcal{B}$ for
all $i\geq 0$.

(2) Without loss of generality, we may assume that $\udim X$ is
minimal in the $\tau$-orbit of $X$. Using the Auslander-Reiten
formula we have $\Hom(X,\tau X)\cong \mathbb{D}\Ext^1(X,X)\neq 0$.
If $\tau X$ has no proper regular submodules, then there is an
epimorphism $X\ra \tau X$, which contradicts the minimality of
$\udim X$. Therefore, $\tau X\notin\mathcal{B}$ and thus
$\tau^iX\notin\mathcal{B}$ for any $i\geq 1$ by (1).

(3) Assume that $X/Y$ is not simple and $N$ is a GR submodule of
$X/Y$. Then the inclusion $N\ra X/Y$ factors through $X$ and thus
$N$ is isomorphic to a proper submodule of $X$. Note that $N$ is
preprojective since $X\in\mathcal{B}$.  On the other hand,  a GR
submodule of a non-simple preinjective module is always a regular
one. Therefore, $X/Y$ is regular.

(4) We may assume that $M$ is not preinjective by the description of the landing part.
If $\mu(M)<\mu(X)$, then   $\mu(P_r)<\mu(M)<\mu(X)$, where $P_r$ is a GR submodule of $X$,
since $M$ is regular. It follows that $|M|>|X|$ by Lemma \ref{GR}(4), which is a contradiction.
\end{proof}

Let $X\in\mathcal{B}$ and $i\geq 1$. Then $\tau^{-i}X\in\mathcal{B}$
by above lemma.  We are able to determine the GR submodules of
$\tau^{-i}X$.

\begin{lemm}\label{sub}Let $X\in\mathcal{B}$ and  $P_r$ a GR submodule of
$X$.
\begin{itemize}
    \item[(1)] If $X/P_r$ is regular, then $\tau^{-i}{P_r}$ is, up to isomorphism, the unique  GR
                 submodule of $\tau^{-i}X$ for each $i\geq 0$.
    \item[(2)] If $X/P_r$ is simple, then $\tau^{-(i-1)}P_{r+1}$ is, up to isomorphism, the unique GR submodule
                 of $\tau^{-i}X$ for each $i\geq 1$.
    \item[(3)] For all $0\leq i<j$, $\mu(\tau^{-i}X)>\mu(\tau^{-j}X)$.
\end{itemize}
\end{lemm}

\begin{proof}(1) If $X/P_r$ is regular, then the GR inclusion
induces a monomorphism $P_{r+2}=\tau^{-1}P_r\ra \tau^{-1} X$ with a
regular factor. If there is a monomorphism $P_{r+3}\ra \tau^{-1}X$,
then there is a monomorphism $P_{r+1}=\tau P_{r+3} \ra X$. This
contradicts $P_r$ is a GR submodule of $X$. Thus $\tau^{-1}P_r$ is a
GR submodule of $\tau^{-1} X$. Since the factor is regular, we have
$\tau^{-i}P_r$ is a GR submodule of $\tau^{-i}X$ for all $i\geq 1$ by induction.

(2) Assume that $X/P_r$ is simple. Let $\udim P_r=(a,b)$. Then
$\udim X=(a+1,b)$. It follows that $\udim \tau^{-1}X=(nb-a-1,
(n^2-1)b-n(a+1))$, $\udim \tau^{-1}P_r=(nb-a, (n^2-1)b-na)$ and
$\udim P_{r+1}=(b,nb-a)$. Comparing the dimension vectors, we know
that $P_{r+1}$ is a GR submodule of $\tau^{-1}X$ (using Lemma
\ref{GR}(3)). Note that $(nb-a-1)-b=(n-1)b-a-1>1$. Thus the GR
factor $\tau^{-i}X/P_{r+1}$ is not simple. It follows from (1) that
$\tau^{-(i-1)}P_{r+1}$ is a GR submodule of $\tau^{-i}X$.

(3) This is straightforward by (1) and (2).
\end{proof}
As a consequence of the last statement of this lemma, we have:

\begin{coro} There does not exist a minimal central measure.
\end{coro}

\begin{proof} For the purpose of a contradiction, we assume that
$M$ is an indecomposable module such that $\mu(M)$ is the minimal
central GR measure.  It follows that $M$ is regular by the
description of the partition and a GR submodule $N$ of $M$ is
preprojective by the minimality of $\mu(M)$. This implies that
$\mu(\tau^{-i}M)<\mu(M)$ for each $i\geq 1$, which is a
contradiction.
\end{proof}

\begin{remark} Note that for a tame quiver, the minimal central measure always
exists \cite{Ch3,Ch4}. However, it does not mean that any wild
quiver has no minimal central measure. For example, let $Q'$ be the
wild quiver with three vertices, labeled by $1,2,3$, and one arrow
from $1$ to $2$ and two arrows from $2$ to $3$.  Then the GR measure
of the indecomposable projective module $P_1$ is
$\mu(P_1)=\{1,3,4\}$, which is the minimal central measure
\cite{Ch6}.
\end{remark}

Let $X\in\mathcal{B}$. Lemma \ref{dim1} and Corollary \ref{corodim1}
give some combinatorial descriptions of the dimension vectors of
$\tau^{-i}X$ for $i\geq 0$. We will use these to study the GR
submodules of $\tau^{-i}X[j]$ for all $i\geq 0$ and $j\geq 2$. We
first recall what a piling submodule is \cite{R5}.

\begin{defi} Let $\Lambda$ be an artin algebra  and $M$ be an indecomposable $\Lambda$-module. Then an
indecomposable submodule $X$ of $M$ is called a piling submodule if
$\mu(X)\geq \mu(Y)$ for all submodules $Y$ of $M$ with $|Y|\leq
|X|$.
\end{defi}
\begin{lemm}[\cite{R5}]\label{piling} 
Let $\Lambda$ be an artin algebra and $M$ be an indecomposable $\Lambda$-module.
Let $X$ be an indecomposable submodule of $M$.
Then $X$ is a piling submodule of $M$ if and only if $\mu(M)$
starts with $\mu(X)$ $($meaning that
$\mu(X)=\mu(M)\cap\{1,2,3,\ldots,|X|\})$.
\end{lemm}

The following result is crucial when calculating the GR submodules
of $\tau^{-i}X[j]$ for $X\in\mathcal{B}$ over $n$-Kronecker quivers.

\begin{prop}\label{keyprop} Let $0\ra X\stackrel{f}{\ra} Y\stackrel{\pi}{\ra} Z\ra 0$ be an short exact sequence
of indecomposable regular modules such that
\begin{itemize}
   \item[(1)] $f$ is an irreducible monomorphism,
   \item[(2)] $Z$ contains a preprojective module as a GR submodule,
   \item[(3)] $|X|<|Z|$.
\end{itemize}
Then $f$ is a GR inclusion. Moreover, $X$ is, up to isomorphism, the
unique GR submodule of $Y$.
\end{prop}

\begin{proof}Let $U\stackrel{g}{\ra} Y$ be an indecomposable regular
submodule. If the composition $\pi g$ is zero, then the inclusion
$g$ factors through $f$ and thus $U$ is isomorphic to a submodule of
$X$. If $\pi g$ is not zero, then  it is an epimorphism since $Z$
contains no proper regular submodules. In particular, $|U|>|Z|$.
Therefore, an indecomposable  proper regular submodule of $Y$ is
either isomorphic to a submodule of $X$, or with length greater than
$|Z|$. Let $V$ be an indecomposable submodule of $M$ such that
$|V|\leq |X|$. If $V$ is regular, then $V$ is isomorphic to a
submodule of $Y$ by above discussion since $|V|\leq |X|<|Z|$. If
$V$ is preprojective, then $\mu(V)<\mu(X)$.  It follows that $X$ is
a piling submodule of $Y$ and thus $\mu(Y)$ starts with $\mu(X)$ by
Lemma \ref{piling}. Let $U$ be a GR submodule of $Y$. Then $U$ is a
regular module. For the purpose of a contradiction, we assume that
$U\ncong X$. Then by above discussion, $|U|>|Z|\geq |X|$. Let
$U_1\subset U_2\subset\ldots\subset U_r=U$ be a GR filtration of
$U$. Since $\mu(X)<\mu(U)<\mu(Y)$, we have $\mu(U)$ starts with
$\mu(X)$. Therefore, there is an $U_i$ such that $|U_i|=|X|$, and
thus $U_i\cong X$.  However, $X\stackrel{f}{\ra}Y$ is an irreducible
monomorphism implies $U$ is decomposable.  This contradiction shows
$X$ is the unique, up to isomorphism, GR submodule of $Y$.
\end{proof}

The following theorem is a direct consequence of Lemma \ref{dim2},
Corollary \ref{corodim1} and Proposition \ref{keyprop}.

\begin{theo}\label{thm1}Let $X\in\mathcal{B}$. Then for each $i\geq 0$ and each $j\geq
1$, the irreducible monomorphism $\tau^{-i}X[j]\ra\tau^{-i}X[j+1]$
is a GR inclusion. Moreover, up to isomorphism, $\tau^{-i}X[j]$ is
the unique GR submodule of $\tau^{-i}X[j+1]$.
\end{theo}

Before ending  this section, we give a description of the dimension
vectors of indecomposable regular modules with the same lengths and
trivial Hom-spaces.

\begin{lemm}\label{hom} Let $X$, $Y$ be indecomposable regular modules with
dimension vectors $(a,b)$ and $(r,s)$, respectively.  Assume that
$|X|=|Y|$, i.e., $a+b=r+s$, and $\Hom(X,Y)=0$. Then $s\geq
b+\frac{q((a,b))}{(n+1)a-b}$.
\end{lemm}

\begin{proof}
Since $\Hom(X,Y)=0$, we have $$\langle(\udim X,\udim
Y\rangle=\dim\Hom(X,Y)-\dimExt^1(X,Y)\leq 0.$$ It follows that
$$ar+bs-nas\leq 0.$$ Using $a+b=r+s$, we obtain that
$a(a+b-s)+bs-nas\leq 0$.  Therefore,
$$((n+1)a-b)s\geq a(a+b).$$
Assume for a contradiction that $(n+1)a\leq b$. Since $(a,b)$ is an
imaginary root,  $\frac{b}{a}\leq\frac{n+\sqrt{n^2-4}}{2}<n$. It
follows that $n+1\leq\frac{b}{a}<n$.
 This contradiction implies
$(n+1)a>b$ and thus
$$s\geq
\frac{a^2+ab}{(n+1)a-b}=b+\frac{a^2-nab+b^2}{(n+1)a-b}=b+\frac{q((a,b))}{(n+1)a-b}.$$
The proof is completed.
\end{proof}

\section{indecomposable modules $\tau^{-i}X$ with  $\udim X=(1,c)$}\label{(1,c)}

For each natural number $1\leq c\leq n-1$, the regular components
containing indecomposable modules  with dimension vectors $(1,c)$
are of special interests. For example, in section \ref{3Kronecker}, we will see that
in case $n=3$, the dimension
vectors of the indecomposable modules in a regular component
containing some $X$ with $\udim X=(1,1)$ or $(1,2)$ relate to pairs
of Fibonacci numbers and the GR measures of the indecomposable
modules in such a component are uniquely determined by their
dimensions.

Let $\udim X=(1,c)$. It is easily seen (for example in the following
lemma) that the GR submodule of $X$ is a projective simple module. Therefore,
$\tau^{-i}X\in\mathcal{B}$ for all $i\geq 1$ by Lemma \ref{B}.
It follows that $\tau^{-i}X[j]$ is, up to isomorphism,  the unique GR submodule
of $\tau^{-i}X[j+1]$ for each $i\geq 0$ and each $j\geq 1$ (Theorem \ref{thm1}). It turns out
that the dimension vector $(1,c)$ and the indexes $i$ and $j$
determine the GR measures. Using this we can show
that $\mu(\tau^{-i}X[j+1])$ is a direct successor of
$\mu(\tau^{-i}X[j])$ for every $i\geq 0$ and $j\geq 1$.

\begin{lemm} \label{c} Let $c$ be a natural number such that $1\leq c\leq n-1$.
Then the vector $(1,c)$ is an imaginary root. Let $X$ be an
indecomposable module with dimension vector $\udim X=(1,c)$.
\begin{itemize}
  \item[(1)]  $X$ is a (regular) quasi-simple module.
  \item[(2)]  A GR submodule of $X$ is  isomorphic to the projective simple module $P_1$.
  \item[(3)]  If $c=1$, $\tau^{-(i-1)}P_2$ is a GR submodule of
               $\tau^{-i}X$ for each $i\geq 1$. If $c>1$, $\tau^{-i}P_1$ is a GR submodule of
                $\tau^{-i}X$ for each $i\geq 0$.
  \item[(4)]  Let  $M$ be an indecomposable module. Then $\mu(M)=\mu(X)$ if and only if
             $\udim M=(1,c)=\udim X$.
\end{itemize}
\end{lemm}

\begin{proof} It is easily seen that $q((1,c))<0$ and thus $(1,c)$
is an imaginary root. Let $X$ be indecomposable with $\udim X=(1,c)$
and $Y$ a GR submodule of $X$. Then $\udim Y=(0,1)$ or $\udim
Y=(1,r)$ with $r<c$. If the second case holds, then the GR factor
has dimension $(0,c-r)$ which is impossible. Thus $Y$ is isomorphic
to $P_1$, the projective simple module. In particular, $X$ is
quasi-simple since it has no proper regular submodules. Thus we may
describe the GR submodules of $\tau^{-i}X$ using Lemma \ref{sub}. If
$M$ is an indecomposable module with $\mu(M)=\mu(X)$, then $P_1$ is
a GR submodule of $M$ and thus there is an epimorphism $P_2\ra M$
(Lemma \ref{GR}). In particular, we have $\udim M=(1,r)$ for some
$r<n$ since $\udim P_2=(1,n)$.  Therefore, $\udim M=\udim X$ since
$|M|=|X|$.
\end{proof}

\begin{lemm} Let $X$ be an indecomposable module with dimension
$\udim X=(1,c)$ and $1\leq c\leq n-1$. Let $i\geq 1$ and suppose
that $\udim\tau^{-i}X=(a,b)$. Then $0<-q((a,b))<(n+1)a-b$.
\end{lemm}
\begin{proof}Since the quadratic $q$ is invariant on the
dimension vectors of the indecomposable modules in a $\tau$-orbit, we have $$-q(\udim
\tau^{-i}X)=-q((a,b))=-q((1,c))=-c^2+nc-1=-(c-\frac{n}{2})^2+\frac{n^2}{4}-1.$$

If $i=1$, then $(a,b)=\udim
\tau^{-1}X=(1,c)\left(\begin{array}{cc}-1&-n\\n&n^2-1\\\end{array}\right)
=(nc-1,(n^2-1)c-n).$ Assume for a contradiction that $-q((a,b))\geq
(n+1)a-b$.  Thus $nc-c^2-1\geq (n+1)(nc-1)-(n^2-1)c+n$.  It follows
that $c^2+c\leq 0$ which is impossible. Thus $-q((a,b))<(n+1)a-b$.

Now we assume that $i\geq 2$. If $i=2$, then
$$(a,b)=\udim\tau^{-2}X=(nc-1,(n^2-1)c-n)\left(\begin{array}{cc}-1&-n\\n&n^2-1\\\end{array}\right).$$
Thus $a-n^2=n^3c-2nc-2n^2+1=nc(n^2-2)-2(n^2-2)-3=(nc-2)(n^2-2)-3\geq
0$ since $c\geq 1$ and $n\geq 3$.  Thus the first coordinate of the
dimension vector $\udim\tau^{-i}X$ is greater than $n^2$ for every
$i\geq 2$ by Lemma \ref{dim1}. Since
$-q(\udim\tau^{-i}X)=-q((a,b))=-(c-\frac{n}{2})^2+\frac{n^2}{4}-1\leq
\frac{n^2}{4}$, it is sufficient to show for each $i\geq 2$ that
$\frac{n^2}{4}<(n+1)a-b$. If $\frac{n^2}{4}\geq (n+1)a-b$, then
$n+1\leq
\frac{n^2}{4a}+\frac{b}{a}<\frac{n+\sqrt{n^2-4}}{2}+\frac{n^2}{4a}<n+\frac{n^2}{4a}<n+1$, since $a\geq n^2$.
This contradiction implies that $0<-q((a,b))<(n+1)a-b$.
\end{proof}

\begin{coro}\label{corohom}
Let $X$ be an indecomposable module with dimension $\udim X=(1,c)$
for some $1\leq c\leq n-1$ and $M=\tau^{-i}X$ for some $i\geq 1$ with
$\udim M=(a,b)$. Let $N$ be an indecomposable regular
module with dimension vector $(r,s)$ such that $|M|=|N|$. If
$\Hom(M,N)=0$, then $s\geq b$
\end{coro}

\begin{proof} We have seen in Lemma \ref{hom} that $s\geq b-\frac{-q((a,b))}{(n+1)a-b}$.
The statement follows since $0<\frac{-q((a,b))}{(n+1)a-b}<1$.
\end{proof}

\begin{lemm}\label{+1}Let $M$ be an indecomposable module with $\udim
M=(a,b)$, $1<a\leq b$. Thus $\udim\tau^{-1}M=(nb-a,(n^2-1)b-na)$.
Assume that that $(a-1,b+1)$ is not an imaginary root. Then neither
is $(nb-a-t,(n^2-1)b-na+t)$ for any $1\leq t\leq nb-a-1$.
\end{lemm}

\begin{proof} Let $t=1$. By assumption $a\leq b$, we have $\frac{a-1}{b+1}\leq
\frac{n-\sqrt{n^2-4}}{2}$, since $(a-1,b+1)$ is not an imaginary
root, and $1<nb-a<(n^2-1)b-na$. Thus we need to show that
$$\frac{(n^2-1)b-na+1}{nb-a-1}=n-\frac{b-n-1}{nb-a-1}\geq
\frac{n+\sqrt{n^2-4}}{2}.$$ Therefore, it is sufficient to show that
$$\frac{b-n-1}{nb-a-1}\leq \frac{a-1}{b+1}.$$
Assume for a contradiction that $\frac{b-n-1}{nb-a-1}\geq
\frac{a-1}{b+1}$. Then $b^2-nb-b+b-n-1\geq nab-a^2-a-nb+a+1$, and
thus $0>q(\udim M)=q((a,b))=b^2-nab+a^2\geq 2+n$, which is impossible.
Since $\frac{(n^2-1)b-na+1}{nb-a-1}\geq \frac{n+\sqrt{n^2-4}}{2}$,
for each $1<t\leq nb-a-1$, we have
$$\frac{(n^2-1)b-na+t}{nb-a-t}>\frac{(n^2-1)b-na+1}{nb-a-1}\geq
\frac{n+\sqrt{n^2-4}}{2}.$$ The proof is completed.
\end{proof}

\begin{coro}\label{coro+1}Let $X$ be an indecomposable module with dimension $\udim X=(1,c)$
with $1\leq c\leq n-1$ and $M=\tau^{-i}X$ for $i\geq 1$ with $\udim
M=(a,b)$. Then $(a-t,b+t)$ is not an imaginary root for any $1\leq
t\leq a-1$.
\end{coro}
\begin{proof} By Lemma \ref{+1}, we need only to show for $i=1$ that
$(a-1,b+1)=(nc-2,(n^2-1)c-n+1)$ is not an imaginary root. It is
sufficient to show that
$$\frac{b+1}{a-1}=\frac{(n^2-1)c-n+1}{nc-2}\geq n>\frac{n+\sqrt{n^2-4}}{2}.$$
If $\frac{(n^2-1)c-n+1}{nc-2}<n$, then we have $n\leq c-1$ which is
impossible.
\end{proof}

\begin{theo}\label{thm2} Let $X$ be an indecomposable module with dimension
$\udim X=(1,c)$ with $1\leq c\leq n-1$ and $M$ an indecomposable
module. Let $i\geq 0$ and $j\geq 1$.  Then
$\mu(M)=\mu(\tau^{-i}X[j])$ if and only if $M\cong \tau^{-i}Y[j]$
for some indecomposable module $Y$ with $\udim Y=(1,c)=\udim X$.
\end{theo}

\begin{proof}  We first assume that $j=1$. By Lemma \ref{c}(4), it is sufficient to consider $i\geq
1$.  If $M\cong \tau^{-i}Y$ for some
indecomposable module $Y$ with $\udim Y=(1,c)=\udim X$, then the GR
measures are obvious the same by Lemma \ref{sub}. Conversely, since
$\mu(M)=\mu(\tau^{-i}X)$, a GR submodule of $M$ is preprojective. In
particular, $M$ has no proper regular submodules. Because
$|M|=|\tau^{-i}X|$, the  vector space $\Hom(\tau^{-i}X,M)=0$ if $M\ncong \tau^{-i}X$. Let
$\udim\tau^{-i}X=(a,b)$  and $\udim M=(r,s)$. Then we have $s\geq b$
by Corollary \ref{corohom} and thus $(r,s)=(a-t,b+t)$ for some
$t\geq 0$. However, Corollary \ref{coro+1} implies that $(r,s)$ is
not an imaginary root if $t\geq 1$. Therefore, $(r,s)=(a,b)$ and
thus $(r,s)\Phi^i=(a,b)\Phi^i=(1,c)$. It follows that
$M\cong\tau^{-i}Y$ for some $Y$ with dimension vector $\udim
Y=(1,c)=\udim X$.

Now we assume that $j>1$. If $M\cong \tau^{-i}Y[j]$ for some
indecomposable module $Y$ with $\udim Y=(1,c)=\udim X$, then
$\tau^{-i}X$ and $\tau^{-i}Y$ have the same dimension vector and
isomorphic GR submodules for each $i\geq 0$. Thus
$\mu(\tau^{-i}Y[j])$ =$\mu(\tau^{-i}X[j])$ since the irreducible
monomorphisms are GR inclusions (Theorem \ref{thm1}).
Conversely, if $\mu(M)=\mu(\tau^{-i}X[j])$, then
$\mu(M)=\mu(\tau^{-i}X[j-1])\cup\{|M|\}$ since $\tau^{-i}X[j-1]$ is
a GR submodule of $\tau^{-i}X[j]$ (Theorem \ref{thm1}). In
particular, if $N$ is a GR submodule of $M$, then
$\mu(N)=\mu(\tau^{-i}X[j-1])$. By induction on $i+j$, we have $N\cong
\tau^{-i}Y[j-1]$ for some indecomposable module $Y$ with $\udim
Y=(1,c)=\udim X$. It follows that $\udim
\tau^{-i}Y[j]=\udim\tau^{-i}X[j]$ and thus
$|\tau^{-i}Y[j]|=|\tau^{-i}X[j]|=|M|$. Note that there is an
epimorphism $\tau^{-i}Y[j]\ra M$ since $N\cong \tau^{-i}Y[j-1]$ is a
GR submodule of $M$ (Lemma \ref{GR}(3)). Therefore, $M\cong
\tau^{-i}Y[j]$.
\end{proof}

This theorem concludes that the GR measures $\mu(\tau^{-i}X[j])$ are
determined by the indexes $i$ and $j$ and the dimension vector
$(1,c)$.  Using this result, we can describe the direct successors
of $\mu(\tau^{-i}X[j])$ for all $i\geq 0$ and $j\geq 1$, where
$\udim X=(1,c)$.

\begin{theo}\label{thm3} Let $X$ be an indecomposable module with dimension
vector $\udim X=(1,c)$ for some $1\leq c\leq n-1$. Then for each $i\geq
0$ and $j\geq 1$, $\mu(\tau^{-i}X[j+1])$ is the direct successor of
$\mu(\tau^{-i}X[j])$.
\end{theo}

\begin{proof} For the purpose of a contradiction, we assume that
$M$ is an indecomposable module such that
$\mu(\tau^{-i}X[j])<\mu(M)<\mu(\tau^{-i}X[j+1])$. It follows that
$\mu(M)=\mu(\tau^{-i}X[j])\cup\{m_1,m_2,\ldots m_t\}$ and
$m_1>|\tau^{-i}X[j+1]|$. Let $N\subset N'$ be indecomposable modules in a
GR filtration of $M$ with $\mu(N)=\mu(\tau^{-i}X[j])$ and
$|N'|=m_1$. Then $N\cong \tau^{-i}Y[j]$ for some indecomposable $Y$
with $\udim Y=\udim X=(1,c)$ by Theorem \ref{thm2}. Since $N$ is a
GR submodule of $N'$, there is an epimorphism $\tau^{-i}Y[j+1]\ra
N'$. It follows that $|N'|=m_1\leq
|\tau^{-i}Y[j+1]|=|\tau^{-i}X[j+1]|$. This is a contradiction.
\end{proof}

Reall that a GR-segment is a sequence of Gabriel-Roiter measures,
 which is closed under direct predecessors
and direct successors. We have proved in \cite{Ch8}
that a tame quiver has only finitely
many GR-segments and conjectured that a wild quiver has
infinitely many GR-segments.
It was already constructed in \cite{Ch5} for $n$-Kronecker
quivers infinitely many GR measures of regular modules,
which admit no direct predecessors. Since each GR measure
that does not admit a direct predecessor produce
a GR segment by taking direct successors,  we have already
infinitely many GR segments for $n$-Kronecker quivers.
Now Theorem \ref{thm3} actually gives a new series of (infinitely many) GR segments.

\begin{theo}There are infinitely many GR-segments of $n$-Kronecker quiver.
\end{theo}

\begin{proof} Fix some $1\leq c\leq n-1$ and an indecomposable module $X$ with
$\udim X=(1,c)$. Let $i\geq 0$.  Starting with $\mu(\tau^{-i}X)$,
we obtains a sequence of GR measures
by taking direct successors
$$\mu(\tau^{-i}X)<\mu(\tau^{-i}X[2])<\mu(\tau^{-i}X[3])<\ldots$$
by Theorem \ref{thm3}.
We may also take direct predecessors of $\mu(\tau^{-i}X)$.
It is easily seen that $\mu(\tau^{-j}X[s])$ never appears
in this sequence for any $j>i$ and $s\geq 1$. It follows that
$\{\mu(\tau^{-i}X[r])\}_{r\geq 1}$ and $\{\mu(\tau^{-j}X[s])\}_{s\geq 1}$
are in different GR segments for all $0<i\neq j$.
Thus there are infinitely many GR-segments.
\end{proof}

\section{$3$-Kronecker quiver}\label{3Kronecker}
It was proved in \cite{Zh} that
indecomposable modules in a regular component of the
Auslander-Reiten quiver of a wild hereditary algebra are uniquely
determined by their dimension vectors. Thus given a regular
component of the Auslander-Reiten quiver of a wild quiver,  there
are only finitely many indecomposable modules with the same length.
Since length is an invariant of  GR measure, it is interesting
to know if these indecomposable modules with the same length have
the same GR measure. However, this is not always the case (see, for example, Section \ref{anti-ex}).

Now we consider a fixed $3$-Kronecker quiver. This quiver is of special interests
because it relates to Fibonacci numbers.  Let $\mathcal{C}$ be a regular component which
contains an indecomposable module with dimension vector $(1,1)$ or $(1,2)$.
We show that the Gabriel-Roiter measures
of the indecomposable modules in $\mathcal{C}$ are uniquely determined by their dimension vectors.

\subsection{Fibonacci numbers and dimension vectors}
We denote by $F_i$ the Fibonacci numbers, which are defined
inductively: $F_0=0$, $F_1=1$ and $F_{n+2}=F_{n+1}+F_n$. Thus we
have the sequence:
$$0,1,1,2,3,5,8,13,21,34,55,89,144,233,377,\ldots$$
With the help of Fibonacci numbers, we may describe the dimension vectors
of indecomposable modules as follows:
\begin{lemm}\label{udimtau} Let $M$ be a non-projective
indecomposable module with dimension vector $(a,b)$.
\begin{itemize}
\item[(1)] If $\tau^i M $ exists for $i>0$, then its dimension vector is
$(F_{4i+2}a-F_{4i}b, F_{4i}a-F_{4i-2}b)$.
\item[(2)] If $\tau^{-i} M $ exists for $i>0$, then its dimension vector is
$(F_{4i}b-F_{4i-2}a, F_{4i+2}b-F_{4i}a)$.
\end{itemize}
\end{lemm}
\begin{proof} We show (1) and (2) follows similarly.
We use induction on $i$.  This is clear for $i=1$. Assume that
$\udim \tau^iM=(F_{4i+2}a-F_{4i}b, F_{4i}a-F_{4i-2}b)$. Then

\begin{displaymath}
\begin{array}{cl}
 & \udim \tau^{i+1}M\\
=& (F_{4i+2}a-F_{4i}b,F_{4i}a-F_{4i-2}b)\left(\begin{array}{cc}8 & 3\\
-3 & -1
\end{array}\right)\\
=& (8(F_{4i+2}a-F_{4i}b)-3(F_{4i}a-F_{4i-2}b),3(F_{4i+2}a-F_{4i}b)-(F_{4i}a-F_{4i-2}b))\\
=& ((8F_{4i+2}-3F_{4i})a-(8F_{4i}-3F_{4i-2})b,
(3F_{4i+2}-F_{4i})a-(3F_{4i}-F_{4i-2})b)\\
\end{array}
\end{displaymath}
Since for each $n\geq 2$, $F_{n+2}=3F_n-F_{n-2}$, i.e., $(F_{n+2},
F_{n})=(F_{n},F_{n-2})\left(\begin{array}{cc}3 & 1\\-1 &
0\end{array}\right)$, we have

\begin{displaymath}\begin{array}{rcl}
(F_{n+6}, F_{n+4}) &=& (F_{n+4},F_{n+2})\left(\begin{array}{cc}3 &
1\\-1 & 0\end{array}\right)\\
& = & (F_{n+2},F_{n})\left(\begin{array}{cc}3 & 1\\-1 &
0\end{array}\right)\left(\begin{array}{cc}3 & 1\\-1 &
0\end{array}\right)\\
& =& (F_{n+2},F_{n})\left(\begin{array}{cc}8 & 3\\
-3 & -1 \end{array}\right)\\
\end{array}\end{displaymath}
Therefore, $\udim \tau^{i+1}M=(F_{4(i+1)+2}a-F_{4(i+1)}b,
F_{4(i+1)}a-F_{4(i+1)-2}b)$.
\end{proof}

\subsection{Regular components containing an indecomposable module
with dimension vector $(1,1)$ or $(1,2)$} First of all, we are able to
 describe the regular components such that a
$\tau$-orbit contains two different indecomposable modules with the
same length. It turns out that up to a scalar such a component is
exactly the one that we have mentioned above.  The following
result was shown in \cite{Ch7} using Fibonacci numbers:

\begin{prop} Let $M$ be an indecomposable regular module
such that $|M|=|\tau^iM|$ for some $i\geq 1$. Then the $\tau$-orbit
contains an indecomposable module with dimension vector $(m,m)$ or
$(m, 2m)$ for some $m\geq 1$.
\end{prop}

The regular components containing  some indecomposable module $X$
with dimension vector $(1,1)$ or $(1,2)$ are of special interests.
On one hand, the dimension vectors of the indecomposable modules in
such a component strongly relate to pairs of Fibonacci numbers. On the other
hand, the indecomposable modules $\tau^{i}X$ (resp.
$\tau^{-i}X$) have no proper regular factors (resp. regular
submodules) for any $i\geq 0$.

In the following, we always denote by $X$ an indecomposable module
with dimension vector $(1,1)$ in a regular component $\mathcal{C}$.
We are going to describe some properties of the dimension vectors of the
indecomposable modules in $\mathcal{C}$.

\begin{remark} All
properties to be presented also hold similarly for a regular
component containing an indecomposable module with dimension vector
$(1,2)$.
\end{remark}

Since indecomposable modules in $\mathcal{C}$ are uniquely
determined by their dimension vectors, we use the dimension vectors to
denote the indecomposable modules.
The following is a part of the regular component $\mathcal{C}$:
$$\xymatrix@C=-12pt@R=8pt{
      &&{\tiny\left(\begin{array}{cc}275 & 110 \\\end{array}\right)}\ar@{->}[rd]
      &&{\tiny\left(\begin{array}{cc}55 & 55 \\\end{array}\right)}\ar@{->}[rd]
      &&{\tiny\left(\begin{array}{cc}110 & 275\\\end{array}\right)}\ar@{->}[rd]
      &&
      \\
       &{\tiny\left(\begin{array}{cc}273 & 105\\\end{array}\right)}\ar@{->}[rd]\ar@{->}[ru]
       &&{\tiny\left(\begin{array}{cc}42 & 21 \\\end{array}\right)}\ar@{->}[rd]\ar@{->}[ru]
       &&{\tiny\left(\begin{array}{cc} 21 & 42\\\end{array}\right)}\ar@{->}[rd]\ar@{->}[ru]
       &&{\tiny\left(\begin{array}{cc}105 & 273\\\end{array}\right)}\ar@{->}[rd]
       &\\
       {\tiny\left(\begin{array}{cc}272 & 104\\\end{array}\right)}\ar@{->}[rd]\ar@{->}[ru]
       &&{\tiny\left(\begin{array}{cc}40 & 16\\\end{array}\right)}\ar@{->}[rd]\ar@{->}[ru]
       &&{\tiny\left(\begin{array}{cc}8 & 8 \\\end{array}\right)}\ar@{->}[rd]\ar@{->}[ru]
       &&{\tiny\left(\begin{array}{cc}16 & 40\\\end{array}\right)}\ar@{->}[rd]\ar@{->}[ru]
       &&{\tiny\left(\begin{array}{cc}104 & 272\\\end{array}\right)}
       \\
       &{\tiny\left(\begin{array}{cc}39 & 15 \\\end{array}\right)}\ar@{->}[rd]\ar@{->}[ru]
       &&{\tiny\left(\begin{array}{cc}6 &  3\\\end{array}\right)}\ar@{->}[rd]\ar@{->}[ru]
       &&{\tiny\left(\begin{array}{cc} 3 & 6\\\end{array}\right)}\ar@{->}[rd]\ar@{->}[ru]
       &&{\tiny\left(\begin{array}{cc}15 & 39\\\end{array}\right)}\ar@{->}[rd]\ar@{->}[ru]
       &\\
       {\tiny\left(\begin{array}{cc}34 & 13\\\end{array}\right)}\ar@{->}[ru]
      &&{\tiny\left(\begin{array}{cc}5 & 2 \\\end{array}\right)}\ar@{->}[ru]
      &&{\tiny\left(\begin{array}{cc}1 & 1 \\\end{array}\right)}\ar@{->}[ru]
      &&{\tiny\left(\begin{array}{cc}2 & 5\\\end{array}\right)}\ar@{->}[ru]
      &&{\tiny\left(\begin{array}{cc}13 & 34\\\end{array}\right)}\\}$$

\begin{lemm}Let $M$ be an indecomposable module.
\begin{itemize}
\item[(1)] If $\udim M=(m,m)$, $m\geq 1$, then $\udim
\tau^iM=(mF_{4i+1},mF_{4i-1})$ and $\udim
\tau^{-i}M=(mF_{4i-1},mF_{4i+1})$, for each $i>0$.
\item[(2)] If $\udim M=(m,2m)$, $m\geq 1$, then $\udim
\tau^{i+1}M=(mF_{4i+3},mF_{4i+1})$ and $\udim
\tau^{-i}M=(mF_{4i+1},mF_{4i+3})$ for each $i\geq 0$.
\end{itemize}
\end{lemm}

\begin{proof} These are direct consequences of Lemma \ref{udimtau}.
\end{proof}

We define inductively a sequence of indecomposable modules in
$\mathcal{C}$. Let $X_1=X$. Assume that $X_n$ is already defined. If
$n$ is odd, then $X_{n+1}$ is the unique indecomposable module with
an irreducible epimorphism $X_{n+1}\ra X_n$; if $n$ is even, then
$X_{n+1}$ is the unique indecomposable module with an irreducible
monomorphism $X_n\ra X_{n+1}$. Thus
$$\xymatrix@C=5pt@R=8pt{
&X_2={\tiny\left(\begin{array}{cc}6 & 3 \\\end{array}\right)}\ar@{->}[ld]\ar@{->}[rd] &&
X_4={\tiny\left(\begin{array}{cc}42 & 21 \\\end{array}\right)}\ar@{->}[ld]\ar@{->}[rd]&\ldots\\
X_1={\tiny\left(\begin{array}{cc}1 & 1 \\\end{array}\right)} &&
X_3={\tiny\left(\begin{array}{cc}8 & 8 \\\end{array}\right)}&&
X_5={\tiny\left(\begin{array}{cc}55 & 55 \\\end{array}\right)}\\
}$$
Note that the quasi-length of $X_n$
is $\ql(M_n)=n$.

\begin{lemm} The dimension vector of $X_n$ is
\begin{displaymath}
\udim X_n=\left\{\begin{array}{ll}F_{2n}(1,1), & $n$\,\ \textrm{is odd}; \\
F_{2n}(2,1), & $n$\,\ \textrm{is
even.}\\\end{array}\right.\end{displaymath}
\end{lemm}

\begin{proof}It is not difficult to see that for each $n\geq 2$ the dimension vector of
$X_n$ is the following:

\begin{displaymath}
\udim X_n=\left\{\begin{array}{ll}\sum_{i=1}^{\frac{n-1}{2}}\udim
\tau^i X+\sum_{i=1}^{\frac{n-1}{2}}\udim \tau^{-i}X+(1,1),&
n\,\ \textrm{is odd};\\
\sum_{i=1}^{\frac{n}{2}}\udim \tau^i
X+\sum_{i=1}^{\frac{n}{2}-1}\udim \tau^{-i}X+(1,1),&
n\,\ \textrm{is even}.\\
\end{array}\right.
\end{displaymath}
Thus if $n$ is odd, then
$$\begin{array}{rcl}\udim X_n &=& (\sum_{i=1}^{\frac{n-1}{2}}F_{4i+1},\sum_{i=1}^{\frac{n-1}{2}}F_{4i-1})
+(\sum_{i=1}^{\frac{n-1}{2}}F_{4i-1},\sum_{i=1}^{\frac{n-1}{2}}F_{4i+1})+(1,1)\\
&=& (\sum_{i=1}^nF_{2i-1},\sum_{i=1}^n F_{2i-1})\\
&=& (F_{2n},F_{2n}).
\end{array}$$
It  follows similarly for $n$  even.
\end{proof}

\begin{coro}Let $M$ be
an indecomposable module in $\mathcal{C}$ with quasi-length $n$. If
$n$ is odd, then the dimension vector of $M$ is
$F_{2n}(F_{4i+1},F_{4i-1})$, $F_{2n}(1,1)$ or
$F_{2n}(F_{4i-1},F_{4i+1})$. If $n$ is even, then $\udim
M=F_{2n}(F_{4i+3},F_{4i+1})$ or $F_{2n}(F_{4i+1},F_{4i+3})$.
\end{coro}

\begin{proof} Since the quasi-length $\ql(M)=n$, $M$ and $X_n$ defined above are in the same $\tau$-orbit.
Thus $M\cong \tau^iX_n$ for some integer $i\in\mathbb{Z}$.
\end{proof}

\subsection{Indecomposable modules with the same length} Now we will
show that in the regular component $\mathcal{C}$, two indecomposable
modules with the same length are in the same $\tau$-orbit. Thus we
may describe their dimension vectors using the properties we have
seen before. For an indecomposable regular module $M$ with quasi-length $\ql=n$, we denote by $\qs(M)$
the unique quasi-simple module $X$ such that $M=X[n]$ and by $\qt(M)$ the unique quasi-simple module
$Y$ such that $M=[n]Y$.

\begin{lemm}Let $M$ and $N$ be two indecomposable modules in
$\mathcal{C}$ with $|M|=|N|$. Then $\ql(M)=\ql(N)$. Thus either
there is an indecomposable module $U$ with dimension vector
$F_{2n}(1,1)$ such that $M\cong \tau^i U$ and $N\cong \tau^{-i}U$
for some $i$, or there is an indecomposable module $V$ with
dimension vectors $F_{2n}(2,1)$ and $\udim \tau^{-1} V=F_{2n}(1,2)$
such that $M\cong\tau^iV$ and $N\cong \tau^{-i}(\tau^{-1}V)$ for
some $i$, where $n=\ql(M)=\ql(N)$.
\end{lemm}
\begin{proof}
The proof depends on a detailed calculation of the dimension
vectors. Let $\qs(M)=M_1$, $\qt(M)=M_2$ and $M_i=\tau^{m_i}X$.
Similarly, let $\qs(N)=N_1$, $\qt(N)=N_2$ and $N_i=\tau^{n_i}X$.
Without loss of generality, we may assume that $m_2\geq n_2$. It is
obvious that $M\cong N$, provided the equality holds.

We first  assume that $m_2>n_2\geq 0$. Then
$$\udim N\leq \sum_{i=0}^{n_1}\udim \tau^iX<\udim \tau^{n_1+1}X$$ by Lemma \ref{dim1}. It
follows that $m_1\leq n_1$. But this implies $\udim M< \udim N$, a
contradiction.

Now we assume that $m_2\geq 0>n_2$. Obviously, we have $m_1\geq
|n_2|$. If $n_1\leq 0$, then $m_1=|n_2|$. Otherwise, $m_1>|n_2|$ and
$$\udim N\leq \sum_{i=0}^{|n_2|}\tau^{-i}X<\udim\tau^{-(|n_2|+1)}X\leq\udim\tau^{-m_1}X,$$
and thus $|N|<|\tau^{-m_1}X|=|\tau^{m_1}X|\leq |M|$, a
contradiction. Since $m_1=|n_2|$, we have $m_2=|n_1|$ and thus
$\ql(M)=\ql(N)$. If $n_1>0$, we have two possibilities $n_1<m_2$ and
$n_1\geq m_2$. In the first case, we have $|n_2|=m_1$. Otherwise,
$|n_2|<m_1$ and thus $\sum_{i=0}^{|n_2|}\udim
\tau^{-i}X<\udim\tau^{-m_1}X$ and
$\sum_{i=1}^{n_1}\udim\tau^iX<\udim\tau^{m_2}X$. It follows that
$|N|<|M|$, which is a contradiction. (Note that here we need $m_1\neq m_2$,
i.e., $M$ is not quasi-simple. If $M$ is quasi-simple, we can
discuss similarly.) In the second case, we have
$|\sum_{i=n_1+1}^{m_1}\tau^iX|=|\sum_{i=-n_2}^{m_2-1}\udim\tau^iX|$.
Then the discussion for the first case applies. The other
possibilities follow similarly.
The proof of the other statements is straightforward.
\end{proof}

Let us denote indecomposable modules in $\mathcal{C}$ by their dimension vectors. 
Given an odd number $n\geq 1$.
There is a short exact sequence
$$0\ra F_{2n}(1,1)\stackrel{f}{\ra}F_{2n+2}(1,2)\ra \tau^{-\frac{n+1}{2}}(1,1)\ra 0.$$
Thus we have short exact sequences
$$0\ra \tau^{-i}F_{2n}(1,1)\stackrel{f_i}{\ra}\tau^{-i}F_{2n+2}(1,2)\ra \tau^{-(i+\frac{n+1}{2})}(1,1)\ra 0$$
where $f_i$ are irreducible monomorphisms.

\begin{lemm}\label{compare3}Let $n\geq 1$ be odd. Then $\udim \tau^{-i}F_{2n}(1,1)<\udim
\tau^{-(i+\frac{n+1}{2})}(1,1)$ for each $i\geq 0$.  Therefore,
$\tau^{-i}F_{2n}(1,1)\stackrel{f_i}{\ra}\tau^{-i}F_{2n+2}(1,2)$ is a
GR inclusion.
\end{lemm}

\begin{proof}
If $i=0$, $$\udim
\tau^{-\frac{n+1}{2}}X=(F_{2n+1},F_{2n+3})>(F_{2n},F_{2n}).$$ Now
assume that $i\geq 1$. Then we need to show
$$(F_{2n}F_{4i-1}, F_{2n}F_{4i+1})<(F_{4(i+\frac{n+1}{2})-1},F_{4(\frac{n+1}{2})+1}).$$
Since  $F_rF_s+F_{r-1}F_{s-1}=F_{r+s-1}$, we get
$F_{2n}F_{4i-1}<F_{2n+4i}<F_{2n+4i+1}$ and
$F_{2n}F_{4i+1}<F_{2n+4i+3}$.
The second statement follows by Proposition \ref{keyprop}.
\end{proof}

Similarly, let $n\geq 2$ be an even number. Then there is a short exact
sequence
$$0\ra F_{2n}(1,2)\stackrel{f}{\ra}\tau^{-1}F_{2n+2}(1,1)\ra \tau^{-(\frac{n}{2}+1)}(1,1)\ra 0.$$
Thus we have short exact sequences
$$0\ra \tau^{-i}F_{2n}(1,2)\stackrel{f_i}{\ra}\tau^{-(i+1)}F_{2n+2}(1,1)\ra \tau^{-(i+\frac{n}{2}+1)}(1,1)\ra 0$$
where $f_i$ are irreducible monomorphisms. As above,  the following
result can be easily shown:
\begin{lemm}\label{compare4}Let $n\geq 2$ be even. Then $\udim \tau^{-i}F_{2n}(1,2)<\udim
\tau^{-(i+\frac{n}{2}+1)}(1,1)$ for each $i\geq 0$.  Therefore,
$\tau^{-i}F_{2n}(1,2)\stackrel{f_i}{\ra}\tau^{-(i+1)}F_{2n+2}(1,1)$
is a GR inclusion.
\end{lemm}

\begin{theo} Let $X$ be an indecomposable module with dimension
vector $(1,1)$ and $\mathcal{C}$ a regular component containing $X$.
Then the GR measures of the indecomposable modules in $\mathcal{C}$
are uniquely determined by their dimension vectors.
\end{theo}

\begin{proof}
By previous discussion,  it is sufficient to consider the following cases:
\begin{itemize}
  \item[(1)] Since for an odd number $n\geq 1$ and each $i\geq 0$,
           $\tau^{-i}F_{2n}(1,1)\stackrel{f_i}{\ra}\tau^{-i}F_{2n+2}(1,2)$ is a
               GR inclusion (Lemma \ref{compare3}), we need to show that
               the length of a GR submodule of $\tau^iF_{2n+2}(2,1)$
               does not equal to $|\tau^{-i}F_{2n}(1,1)|$.
  \item[(2)] Since for an even number $n\geq 2$ and each $i\geq 0$,
           $\tau^{-i}F_{2n}(1,2)\stackrel{f_i}{\ra}\tau^{-(i+1)}F_{2n+2}(1,1)$
         is a GR inclusion (Lemma \ref{compare4}), we need to show that the length of a GR
         submodule of $\tau^{i+1}F_{2n+2}(1,1)$ does not equal to
         $|\tau^{-i}F_{2n}(1,2)|$.
\end{itemize}
We show (1) and (2) follows similarly. A GR submodule $Y$ of
$\tau^iF_{2n+2}(2,1)$ is obviously a regular module. Assume that
$\udim Y=(a,b)$. Then $\udim \tau^{-1}Y=(3b-a,8b-3a)$. Let $M$ be
the unique indecomposable module with an irreducible monomorphism
$Y\ra M$. Then there is an epimorphism $M\ra\tau^iF_{2n+2}(2,1)$
(Lemma \ref{GR}(3)). Note that $\udim M\leq(3b,9b-3a)$. Assume
that $|Y|=|\tau^{-i}F_{2n}(1,1)|$. Then we have
$$a+b=F_{2n}(F_{4i-1}+F_{4i+1}) \quad \textrm{and}\quad 3b\geq F_{2n+2}F_{4i+3}.$$
The second inequality follows because
$\udim\tau^iF_{2n+2}(2,1)=F_{2n+2}(F_{4i+3},F_{4i+1})$. Therefore,
$$\frac{a+b}{3b}\leq \frac{F_{2n}(F_{4i-1}+F_{4i+1})}{F_{2n+2}F_{4i+3}}$$
and thus
$$\frac{a}{b}\leq \frac{3F_{2n}}{F_{2n+2}}\frac{(F_{4i-1}+F_{4i+1})}{F_{4i+3}}-1.$$
For the purpose of a contradiction, we show that the right hand side
is smaller than $\frac{3-\sqrt{5}}{2}=2-\varphi$ where
$\varphi=\frac{1+\sqrt{5}}{2}$ is the golden ratio. If this is the
case, then $(a,b)$ is not a root. Thus there does not exist an
indecomposable module with dimension vector $(a,b)$ and we obtain a
contradiction.

We simply write $A=\frac{(F_{4i-1}+F_{4i+1})}{F_{4i+3}}$. It is
sufficient to show
$$\frac{F_{2n}}{F_{2n+2}}A<\frac{3-\varphi}{3}.$$ Using
$F_m=\frac{\varphi^m-(1-\varphi)^m}{\sqrt{5}}$, we may easily obtain
that $\varphi F_m=F_{m+1}-(1-\varphi)^m$. Thus
\begin{displaymath}
\begin{array}{rcl}
(1+\varphi)F_m=\varphi^2F_m &=& \varphi F_{m+1}-(1-\varphi)^m\varphi\\
                            &=& F_{m+2}-(1-\varphi)^{m+1}-(1-\varphi)^m\varphi\\
                            &=& F_{m+2}-(1-\varphi)^m(1-\varphi+\varphi)\\
                            &=& F_{m+2}-(1-\varphi)^m.\\
\end{array}
\end{displaymath}
Replacing $m$ by $2n$, we get
$\frac{F_{2n}}{F_{2n+2}}<\frac{F_{2n}}{(1+\varphi)F_{2n}}=\frac{1}{1+\varphi}$.
Thus it is sufficient to show that
$$\frac{1}{(1+\varphi)}A<\frac{3-\varphi}{3}.$$
Note that $\frac{(3-\varphi)(1+\varphi)}{3}=\frac{2+\varphi}{3}>1$.
However,
$$A=\frac{(F_{4i-1}+F_{4i+1})}{F_{4i+3}}<\frac{(F_{4i+2}+F_{4i+1})}{F_{4i+3}}=1.$$
The proof is finished.
\end{proof}

\begin{remark} The theorem can be generalized for regular components over $n$-Kronecker quivers,
which contains an indecomposable module with dimension vector $(1,1)$ or $(1,n-1)$.
\end{remark}

\subsection{A counter example}\label{anti-ex}

In the following example, we will see that non-isomorphic indecomposable modules
in a regular component may have the same GR measure for some wild quiver.
\begin{example} Let $k$ be an algebraically closed field and
$Q=(Q_0,Q_1)$ be a tame quiver of type
$\widetilde{A}_n$ with $n\geq 3$ an odd number and with sink-source
orientation, i.e., a vertex in $Q_0$ is either a sink or a source.
Without loss of generality, we may certainly assume that the
vertices in $Q_0$ are labeled by $\{a_1,a_2,\ldots,a_{n+1}\}$ and
there is an arrow $a_1\ra a_2$. This means that $a_1$ is a source.
Let $\overline{Q}$ be the one point extension of $Q$ with respect to
the indecomposable projective module $P_{a_1}$. More precisely,
$\overline{Q}_0=Q_0\cup \{a_0\}$ and $\overline{Q}_1=Q_1\cup\{a_0\ra
a_1\}$.  For example, if $n=3$, then $\overline{Q}$ is the
following:
$$
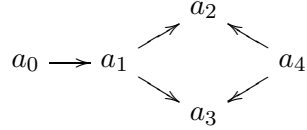
\xymatrix@C=15pt@R=6pt{&&a_2&\\a_0\ar[r] & a_1\ar[ru]\ar[rd] &&
a_4\ar[lu]\ar[ld]\\ && a_3&\\ }$$ We know from the structure of the
Auslander-Reiten quiver of $Q$ that there are two exceptional
regular tubes, each of which contains precisely $\frac{n+1}{2}$
non-isomorphic indecomposable modules of length $2$ as quasi-simple
modules. Let $M$ be  one of those  with dimension vector
$$(\udim
M)_{a_i}=\left\{\begin{array}{ll}1, & i=1, 2;\\0, &
\textrm{otherwise}.\\\end{array}\right.$$  Then as $kQ$-modules,
$M,\tau_Q M,\ldots,\tau_Q^{\frac{n-1}{2}}M$ are pairwise
non-isomorphic quasi-simple regular modules in a regular tube. It is
not difficult to see that $\tau_{\overline{Q}}^iM=\tau_Q^iM$ for
every $0\leq i\leq\frac{n-1}{2}$ and
$\tau_{\overline{Q}}^{\frac{n+1}{2}}M$ is an indecomposable module
with length $3$. Thus $\tau_{\overline{Q}}^{i}M$ are quasi-simple
regular modules with length $2$ for all $0\leq i\leq \frac{n-1}{2}$.
Obviously, they all have the same GR measure $\{1,2\}$.
\end{example}

\end{document}